\documentclass[psamsfonts]{amsart}

\usepackage{amssymb,amsfonts,amscd}
\usepackage[all,arc]{xy}
\usepackage{enumerate}
\usepackage{mathrsfs}

\newtheorem{thm}{Theorem}[section]
\newtheorem{cor}[thm]{Corollary}
\newtheorem{prop}[thm]{Proposition}
\newtheorem{lem}[thm]{Lemma}

\newtheorem*{thm1}{Theorem 1}

\theoremstyle{definition}

\newtheorem{exmp}[thm]{Example}

\theoremstyle{remark}

\makeatletter
\let\c@equation\c@thm
\makeatother
\numberwithin{equation}{section}

\newcommand{\upperRomannumeral}[1]{\uppercase\expandafter{\romannumeral#1}}

\title[]{A proof of Milnor conjecture in dimension 3}

\author[]{Jiayin Pan}

\begin{document}

\begin{abstract}
	We present a proof of Milnor conjecture in dimension $3$ based on Cheeger-Colding
	theory on limit spaces of manifolds with Ricci curvature bounded below. It is different from \cite{Liu}
	that relies on minimal surface theory.
\end{abstract}

\maketitle

\section{Introduction}

Milnor \cite{Mi} in $1968$ conjectured that any open $n$-manifold $M$ with $\mathrm{Ric}_M\ge 0$ has a finitely generated fundamental group. This conjecture remains open today. It was proven for manifolds with Euclidean volume growth by Anderson \cite{An} and Li \cite{Li} independently, and manifolds with small diameter growth by Sormani \cite{Sor}. For background and relevant examples regarding Milnor conjecture, see \cite{SS}.

For $3$-manifolds, Schoen and Yau \cite{SY} developed minimal surfaces theory in dimension $3$ and proved that any $3$-manifold of positive Ricci curvature is diffeomorphic to $\mathbb{R}^3$. Recently, based on minimal surface theory, Liu \cite{Liu} proved that any $3$-manifold with $\mathrm{Ric}\ge 0$ either is diffeomorphic to $\mathbb{R}^3$ or its universal cover splits. In particular, this confirms Milnor conjecture in dimension $3$.

There are some interests to find a proof of Milnor conjecture in dimension $3$ not relying on minimal surface theory. Our main attempt is to accomplish this by using structure results for limits spaces of manifolds with Ricci curvature bounded below \cite{Co,CC1,CC2,CN}, equivariant Gromov-Hausdorff convergence \cite{FY} and pole group theorem \cite{Sor}.

\begin{thm1}\label{3}
	Let $M$ be an open $3$-manifold with $\mathrm{Ric}_M\ge 0$, then $\pi_1(M)$ is finitely generated.
\end{thm1}

For any open $3$-manifold $M$ of $\mathrm{Ric}_M\ge 0$ and any sequence $r_i\to\infty$, by Gromov's precompactness theorem \cite{Gro2}, we can pass to some subsequences and consider tangent cones at infinity of $M$ and its Riemannian universal cover $\widetilde{M}$ coming from the sequence $r_i^{-1}\to 0$:
\begin{center}
	$\begin{CD}
	(r^{-1}_i\widetilde{M},\tilde{p}) @>GH>> 
	(C_\infty\widetilde{M},\tilde{o})\\
	@VV\pi V @VV\pi V\\
	(r^{-1}_iM,p) @>GH>> (C_\infty M,o).
	\end{CD}$
\end{center}
We roughly illustrate our approach to prove Theorem \ref{3}. If $\pi_1(M,p)$ is not finitely generated, then we draw a contradiction by choosing some sequence $r_i\to\infty$ and eliminating all the possibilities regarding the dimension of $C_\infty \widetilde{M}$ and $C_\infty M$ above in the Colding-Naber sense \cite{CN}, which are integers $1,2$ or $3$.
 
We also make use of some reduction results by Wilking \cite{Wi} and Evans-Moser \cite{EM}. The first reduces any non-finitely generated fundamental groups to abelian ones in any dimension, while the latter further reduces abelian non-finitely generated ones to some subgroup of the additive group of rationals in dimension $3$. In particular, we can assume that $\pi_1(M)$ is torsion free if it is not finitely generated. One observation is that, if $\pi_1(M,p)$ is torsion free, then in the space $(C_\infty\widetilde{M},\tilde{v},G)$ above, the orbit $G\cdot\tilde{v}$ is not discrete (See Corollary \ref{non_dis_orb_cor}). This observation plays a key role in the proof.

The author would like to thank Professor Xiaochun Rong and Professor Jeff Cheeger for suggestions during the preparation of this note.

\section{Proof of Theorem \ref{3}}

We start with the following reductions by Wilking and Evan-Moser.

\begin{thm}\cite{Wi}\label{red_W}
	Let $M$ be an open manifold with $\mathrm{Ric}_M\ge 0$.
	If $\pi_1(M)$ is not finitely generated, then it contains a non-finitely generated abelian subgroup.
\end{thm}

\begin{thm}\cite{EM}\label{red_EM}
	Let $M$ be a $3$-manifold. If $\pi_1(M)$ is abelian and not finitely generated, then $\pi_1(M)$ is torsion free.
\end{thm}

Evans-Moser \cite{EM} actually showed that $\pi_1(M)$ is a subgroup of the additive group of rationals. Being torsion free is sufficient for us to prove Theorem \ref{3}.

Gromov \cite{Gro1} introduced the notion of short generators of $\pi_1(M,p)$. By path lifting, $\pi_1(M,p)$ acts on $\widetilde{M}$ isometrically. We say that $\{\gamma_1,...,\gamma_i,...\}$ is a set of short generators of $\pi_1(M,p)$, if
\begin{center}
$d(\gamma_1\tilde{p},\tilde{p})\le d(\gamma\tilde{p},\tilde{p})$ for all $\gamma\in\pi_1(M,p)$,
\end{center}
and for each $i$,
\begin{center}
$d(\gamma_i\tilde{p},\tilde{p})\le d(\gamma\tilde{p},\tilde{p})$ for all $\gamma\in\pi_1(M,p)-\langle\gamma_1,...,\gamma_{i-1}\rangle$,
\end{center}
where $\langle\gamma_1,...,\gamma_{i-1}\rangle$ is the subgroup generated by $\gamma_1,...,\gamma_{i-1}$.

Let $M$ be an open $3$-manifold with $\mathrm{Ric}_M\ge 0$. We always denote $\pi_1(M,p)$ by $\Gamma$. Suppose that $\Gamma$ is not finitely generated, then by Theorems \ref{red_W} and \ref{red_EM}, we can assume that $\Gamma$ is torsion free. Let $\{\gamma_1,...,\gamma_i,...\}$ be an infinite set of short generators at $p$. Since $\Gamma$ is a discrete group acting freely on $\widetilde{M}$, we have $r_i=d(\tilde{p},\gamma_i\tilde{p})\to \infty$. When considering a tangent cone at infinity of $\widetilde{M}$ coming from the sequence $r_i^{-1}\to 0$, we also take $\Gamma$-action into account. Passing to some subsequences if necessary, we assume the following sequences converge in equivariant Gromov-Hausdorff topology \cite{FY}:
\begin{center}
	$\begin{CD}
	(r^{-1}_i\widetilde{M},\tilde{p},\Gamma) @>GH>> 
	(\widetilde{Y},\tilde{y},G)\\
	@VV\pi V @VV\pi V\\
	(r^{-1}_iM,p) @>GH>> (Y=\widetilde{Y}/G,y).
	\end{CD}$
	$(\star)$
\end{center}
Colding-Naber \cite{CN} showed that the isometry group of any Ricci limit space is a Lie group. In particular, $G$ above, as a closed subgroup of $\mathrm{Isom}(\widetilde{Y})$, is a Lie group. 

We recall the dimension of Ricci limit spaces in the Colding-Naber sense \cite{CN}. A point $x$ in some Ricci limit space $X$ is $k$-regular, if any tangent cone at $x$ is isometric to $\mathbb{R}^k$. Colding-Naber showed that there is a unique $k$ such that $\mathcal{R}_k$, the set of $k$-regular points, has full measure in $X$ with respect to any limit renormalized measure (See \cite{CC2,CN}). We regard such $k$ as the dimension of $X$ and denote it by $\dim(X)$. It is unknown whether in general the Hausdorff dimension of $X$ equals to $\dim(X)$. For Ricci limit spaces coming from $3$-manifolds, dimension in the Colding-Naber sense equals to Hausdorff dimension, which follows from Theorem 3.1 in \cite{CC2} and \cite{Hon1}.

As indicated in the introduction, we prove Theorem \ref{3} by eliminating all possibilities regarding the dimension of $Y$ and $\widetilde{Y}$ in ($\star$). There are three possibilities and we rule out each of them, which finishes the proof of Theorem \ref{3}.\\
\noindent\textit{Case 1. $\dim(\widetilde{Y})=3$} (Lemma \ref{not_3});\\
\textit{Case 2. $\dim(Y)=\dim(\widetilde{Y})=2$} (Lemma \ref{not_2});\\
\textit{Case 3. $\dim(Y)=1$} (Lemma \ref{not_1}).\\

\begin{lem}\label{non_dis_orb}
	Let $(M_i,p_i)$ be a sequence of complete $n$-manifolds and $(\widetilde{M}_i,\tilde{p}_i)$ be their universal covers. Suppose that the following sequence converges
	$$(\widetilde{M}_i,\tilde{p}_i,\Gamma_i)\overset{GH}\longrightarrow(\widetilde{X},\tilde{p},G),$$
	where $\Gamma_i=\pi_1(M_i,p_i)$ is torsion free for each $i$. If the orbit $G\cdot\tilde{p}$ is discrete in $\widetilde{X}$, then there is $N$ such that 
	$$\#\Gamma_i(1)\le N$$
	for all $i$, where $\#\Gamma_i(1)$ denotes the number of elements in
	$$\Gamma_i(1)=\{\gamma\in \Gamma_i\ |\ d(\gamma\tilde{p}_i,\tilde{p}_i)\le 1\}.$$
\end{lem}

\begin{proof}
	We claim that if a sequence $\gamma_i\in\Gamma_i$ such that $\gamma_i\overset{GH}\to g\in G$ with $g$ fixing $\tilde{p}$, then $g=e$, the identity element, and $\gamma_i=e$ for all $i$ sufficiently large. In fact, suppose that $\gamma_i\not= e$ for some subsequence. Since $\gamma_i$ is torsion free, we always have $\mathrm{diam}(\langle\gamma_i\rangle\cdot \tilde{p}_i)=\infty$. Together with $d(\gamma_i\tilde{p}_i,\tilde{p}_i)\to 0$, we see that $G\cdot\tilde{p}$ can not be discrete, which contradicts with the assumption.
	
	Therefore, there exists $i_0$ large such that for all $g\in G(2)$ and any two sequences with $\gamma_i\overset{GH}\to g$ and $\gamma'_i\overset{GH}\to g$, $\gamma_i=\gamma'_i$ holds for all $i\ge i_0$. In particular, we conclude that
	$$\#\Gamma_i(1)\le \# G(2)<\infty$$
	for all $i\ge i_0$.
\end{proof}

\begin{cor}\label{non_dis_orb_cor}
	Let $(M,p)$ be an open $n$-manifold with $\mathrm{Ric}_M\ge 0$ and $(\widetilde{M},\tilde{p})$ be its universal cover. Suppose that $\Gamma=\pi_1(M,p)$ is torsion free, then for any $s_i\to \infty$ and any convergent sequence
	$$(s_i^{-1}\widetilde{M},p,\Gamma)\overset{GH}\longrightarrow(C_\infty\widetilde{M},\tilde{o},G),$$
	the orbit $G\cdot\tilde{v}$ is not discrete.
\end{cor}

\begin{proof}
	The proof follows directly from Lemma \ref{non_dis_orb}. If $G\cdot\tilde{o}$ is discrete, then there is $N$ such that $\#\Gamma(s_i)\le N$ for all $i$. On the other hand, $\#\Gamma(s_i)\to \infty$ because $\Gamma$ is torsion free. A contradiction. 
\end{proof}

\begin{lem}\label{non_cnt_orb}
	Let $(M,p)$ be an open $n$-manifold with $\mathrm{Ric}_M\ge 0$ and $(\widetilde{M},\tilde{p})$ be its universal cover. Suppose that $\Gamma=\pi_1(M,p)$ has infinitely many short generators $\{\gamma_1,...,\gamma_i,...\}$. Then in the following tangent cone at infinity of $\widetilde{M}$
	$$(r_i^{-1}\widetilde{M},p,\Gamma)\overset{GH}\longrightarrow(\widetilde{Y},\tilde{y},G),$$
	the orbit $G\cdot\tilde{y}$ is not connected, where $r_i=d(\gamma_i\tilde{p},\tilde{p})\to\infty$.
\end{lem}

\begin{proof}
	On $r_i^{-1}\widetilde{M}$, $\gamma_i$ has displacement $1$ at $\tilde{p}$. By basic properties of short generators, $\gamma_i\tilde{p}$ has distance $1$ from the orbit $H_i\cdot\tilde{p}$, where $H_i=\langle\gamma_1,...,\gamma_{i-1}\rangle$. From equivariant convergence
	$$(r^{-1}_i\widetilde{M},\tilde{p},H_i,\gamma_i)\overset{GH}\longrightarrow(\widetilde{Y},\tilde{y},H,g),$$
	we conclude $d(g\tilde{y},H\cdot\tilde{y})=1$. It is obvious that $H$ contains $G_0$, the connected component of $G$ containing the identity. Thus $d(g\tilde{y},G_0\cdot\tilde{y})\ge 1$ and the orbit $G\cdot \tilde{y}$ is not connected.
\end{proof}

We recall cone splitting principle, which follows from splitting theorem for Ricci limit spaces \cite{CC1}.

\begin{prop}\label{cone_split}
	Let $(X,p)$ be the limit of a sequence of complete $n$-manifolds $(M_i,p_i)$ of $\mathrm{Ric}_{M_i}\ge 0$. Suppose that $X=\mathbb{R}^k\times C(Z)$ is a Euclidean cone with vertex $p=(0,z)$. If there is an isometry $g\in \mathrm{Isom}(X)$ with $g(0,z)\not\in \mathbb{R}^k\times \{z\}$, then $X$ splits isometrically as $\mathbb{R}^{k+1}\times C(Z')$.
\end{prop}

\begin{lem}\label{not_3}
	Case 1 can not happen.
\end{lem}

\begin{proof}
	When $\dim(\widetilde{Y})=3$, $\widetilde{Y}$ is a non-collapsing limit space \cite{CC2}, that is, there is $v>0$ such that
	$$\mathrm{vol}(B_1(\tilde{p},r_i^{-1}\widetilde{M}))\ge v$$
	for all $i$. By relative volume comparison, this implies that $\widetilde{M}$ has Euclidean volume growth
	$$\lim\limits_{r\to\infty}\dfrac{\mathrm{vol}(B_r(\tilde{p}))}{r^n}\ge v.$$
	By \cite{CC2}, $\widetilde{Y}$ is a Euclidean cone $\mathbb{R}^k\times C(Z)$ with vertex $\tilde{y}=(0,z)$, where $C(Z)$ does not contain any line and $z$ is the vertex of $C(Z)$. We rule out all the possibilities of $k\in \{0,1,2,3\}$. 
	
	If $k=3$, then $\widetilde{Y}=\mathbb{R}^3$. Thus $\widetilde{M}$ is isometric to $\mathbb{R}^3$ \cite{Co}.
	
	If $k=2$, then according to co-dimension $2$ \cite{CC2}, actually $\widetilde{Y}=\mathbb{R}^3$.
	
	If $k=1$, then $Y=\mathbb{R}\times C(Z)$. By Proposition \ref{cone_split}, the orbit $G\cdot \tilde{y}$ is contained in $\mathbb{R}\times\{z\}$. Applying Lemma \ref{non_cnt_orb}, we see that $G\cdot \tilde{y}$ is not connected. Note that a non-connected orbit in $\mathbb{R}$ is either a $\mathbb{Z}$-translation orbit, or a $\mathbb{Z}_2$-reflection orbit. In particular, the orbit $G\cdot\tilde{y}$ must be discrete. This contradicts with Corollary \ref{non_dis_orb_cor}.
	
	If $k=0$, then $Y=C(Z)$ with no lines. Again by Proposition \ref{cone_split}, the orbit $G\cdot\tilde{y}$ must be a single point $\tilde{y}$, which is a contradiction to Lemma \ref{non_cnt_orb}.
\end{proof}

\begin{lem}\label{not_2}
	Let $(M,p)$ be an open $n$-manifold with $\mathrm{Ric}_M\ge 0$ and $(\widetilde{M},\tilde{p})$ be its universal cover. Assume that $\Gamma=\pi_1(M,p)$ is torsion free. Then for any $s_i\to \infty$ and any convergent sequence
	\begin{center}
		$\begin{CD}
		(s^{-1}_i\widetilde{M},\tilde{p},\Gamma) @>GH>> 
		(C_\infty\widetilde{M},\tilde{o},G)\\
		@VV\pi V @VV\pi V\\
		(s^{-1}_iM,p) @>GH>> (C_\infty M,o),
		\end{CD}$
	\end{center}
	$\dim(C_\infty\widetilde{M})=\dim(C_\infty M)$ can not happen. In particular, Case 2 can not happen.
\end{lem}

\begin{proof}[Proof of Lemma \ref{not_2}]
	We claim that $G$ is a discrete group when $\dim({C_\infty\widetilde{M}})=\dim(C_\infty M)=k$. If the claim holds, then the desired contradiction follows from Corollary \ref{non_dis_orb_cor}.
	
	It remains to verify the claim. Suppose that $G_0$ is non-trivial, then we pick $g\not=e$ in $G_0$. Note that there is a $k$-regular point $\tilde{q}\in C_\infty \widetilde{M}$ such that $d(g\tilde{q},\tilde{q})>0$ and $\tilde{q}$ projects to a $k$-regular point $q\in C_\infty M$. In fact, let $\mathcal{R}_k(C_\infty M)$ be the set of $k$-regular points in $C_\infty M$. Since $\mathcal{R}_k(C_\infty M)$ is dense in $C_\infty M$, its pre-image $\pi^{-1}(\mathcal{R}_k(C_\infty M))$ is also dense in $C_\infty \widetilde{M}$. Let $\tilde{q}$ be a point in the pre-image such that $d(g\tilde{q},\tilde{q})>0$. Note that any tangent cone at $\tilde{q}$ splits $\mathbb{R}^k$-factor isometrically. By Proposition 3.78 in \cite{Hon2} (also see Corollary 1.10 in \cite{KL}), it follows that any tangent cone at $\tilde{q}$ is isometric to $\mathbb{R}^k$. In other words, $\tilde{q}$ is $k$-regular.    
	
	Along a one-parameter subgroup of $G_0$ containing $g$, we can choose a sequence of elements $g_j\in G_0$ with $d(g_j\tilde{q},\tilde{q})=1/j\to 0$. We consider a tangent cone at $\tilde{y}$ and $y$ respectively coming from the sequence $j\to\infty$. Passing to some subsequences if necessary, we obtain
	\begin{center}
		$\begin{CD}
		(jC_\infty\widetilde{M},\tilde{q},G,g_j) @>GH>> 
		(C_{\tilde{q}} C_\infty\widetilde{M},\tilde{o}',H, h)\\
		@VV\pi V @VV\pi V\\
		(jC_\infty M,q) @>GH>> (C_qC_\infty M,o').
		\end{CD}$
	\end{center}
	with $C_{\tilde{q}} C_\infty\widetilde{M}/H=C_qC_\infty M$ and $d(h \tilde{o}',\tilde{o}')=1$. On the other hand, since both $q$ and $\tilde{q}$ are $k$-regular, $C_{\tilde{q}} C_\infty\widetilde{M}=C_qC_\infty M=\mathbb{R}^k$. This is a contradiction to $H\not=\{e\}$. Hence the claim holds. 
\end{proof}

To rule out the last case $\dim(Y)=1$, we recall Sormani's pole group theorem \cite{Sor}. We say that a length space $X$ has a pole at $x\in X$, if for all $y\not=x$, there is a ray starting from $x$ and going through $y$.

\begin{thm}\cite{Sor}\label{non_polar}
	Let $(M,p)$ be an open $n$-manifold with $\mathrm{Ric}_M\ge 0$ and $(\widetilde{M},\tilde{p})$ be its universal cover. Suppose that $\Gamma=\pi_1(M,p)$ has infinitely many short generators $\{\gamma_1,...,\gamma_i,...\}$. Then in the following tangent cone at infinity of $M$
	$$(r_i^{-1}M,p)\overset{GH}\longrightarrow(Y,y),$$
	$Y$ can not have a pole at $y$, where $r_i=d(\gamma_i\tilde{p},\tilde{p})\to\infty$.
\end{thm}

\begin{lem}\label{not_1}
	Case 3 can not happen.
\end{lem}

\begin{proof}
	By \cite{Hon1} (also see \cite{Chen}), $Y$ is a topological manifold of dimension $1$. Since $Y$ is non-compact, $Y$ is either a line $(-\infty,\infty)$ or a half line $[0,\infty)$. By Theorem \ref{non_polar}, $Y$ can not have a pole at $y$. Thus there is only one possibility left: $Y=[0,\infty)$ but $y$ is not the endpoint $0\in [0,\infty)$. Put $d=d_Y(0,y)>0$. We rule out this case by a rescaling argument and Lemmas \ref{not_3}, \ref{not_2} above. (In general, it is possible for an open manifold having a tangent cone at infinity as $[0,\infty)$ with base point not being $0$. See example \ref{tree}.) 
	
	Let $\alpha(t)$ be a unit speed ray in $M$ starting from $p$ and converging to the unique ray from $y$ in $Y=[0,\infty)$ with respect to the sequence $(r_i^{-1}M,p)\overset{GH}\longrightarrow(Y,y)$. Let $x_i\in r_i^{-1}M_i$ be a sequence of points converging to $0\in Y$, then $r_i^{-1}d_M(p,x_i)\to d$. For each $i$, let $c_i(t)$ be a minimal geodesic from $x_i$ to $\alpha(dr_i)$, and $q_i$ be a closest point to $p$ on $c_i$. We reparametrize $c_i$ so that $c_i(0)=q_i$. With respect to $(r_i^{-1}M,p)\overset{GH}\longrightarrow(Y,y)$, $c_i$ subconverges to the unique segment between $0$ and $2d\in [0,\infty)$. Clearly,
	$$r_i^{-1}d_M(x_i,\alpha(dr_i))\to 2d, \quad r_i^{-1}d_i\to 0,$$
	where $d_i=d_M(p,c_i(0))$.
	
	If $d_i\to\infty$, then we rescale $M$ and $\widetilde{M}$ by $d_i^{-1}\to 0$. Passing to some subsequences if necessary, we obtain
	\begin{center}
		$\begin{CD}
		(d^{-1}_i\widetilde{M},\tilde{p},\Gamma) @>GH>> 
		(\widetilde{Y}',\tilde{y}',G')\\
		@VV\pi V @VV\pi V\\
		(d^{-1}_iM,p) @>GH>> (Y',y').
		\end{CD}$
	\end{center}
	If $\dim(Y')=1$, then we know that $Y'=(-\infty,\infty)$ or $[0,\infty)$. On the other hand, since
	$$d_i^{-1}d_M(c_i(0),x_i)\to\infty,\quad d_i^{-1}d_M(c_i(0),\alpha(dr_i))\to\infty, \quad d_i^{-1}d_M(c_i,p)=1,$$
	$c_i$ subconverges to a line $c_\infty$ in $Y'$ with $d(c_\infty,y')=1$. Clearly this can not happen in $(-\infty,\infty)$ nor $[0,\infty)$. If $\dim(\widetilde{Y}')=3$, then $\widetilde{M}$ has Euclidean volume growth and thus $\dim(\widetilde{Y})=3$. This case is already covered in Lemma \ref{not_3}. The only situation left is $\dim(\widetilde{Y}')=\dim(Y')=2$. By Lemma \ref{not_2}, this also leads to a contradiction. In conclusion, $d_i\to\infty$ can not happen.
	
	If there is some $R>0$ such that $d_i\le R$ for all $i$, then on $M$, $c_i$ subconverges to a line $c$ with $c(0)\in B_{2R}(p)$. Consequently, $M$ splits a line isometrically \cite{CG}, which contradicts with $Y=[0,\infty)$. This completes the proof.
\end{proof} 

\begin{exmp}\label{tree}
	We construct a surface $(S,p)$ isometrically embedded in $\mathbb{R}^3$ such that $S$ has a tangent cone at infinity as $[0,\infty)$, but $p$ does not correspond to $0$. We first construct a subset of $xy$-plane by gluing intervals. Let $r_i\to\infty$ be a positive sequence with $r_{i+1}/r_i\to\infty$. Starting with a interval $I_1=[-r_1,r_2]$, we attach a second interval $I_2=[-r_3,r_4]$ perpendicularly to $I_1$ by identifying $r_2\in I_1$ and $0\in I_2$. Repeating this process, suppose that $I_{k}$ is attached, then we attach the next interval $I_{k+1}=[-r_{2k+1},r_{2k+2}]$ perpendicularly to $I_k$ by identifying $r_{2k}\in I_k$ and $0\in I_{k+1}$. In the end, we get a subset $T$ in the $xy$-plane consisting of segments. We can smooth the $\epsilon$-neighborhood of $T$ in $\mathbb{R}^3$ so that it has sectional curvature $\ge -C$ for some $\epsilon,C>0$. We call this surface $S$ and let $p\in S$ be a point closest to $0\in I_1$ as base point. If we rescale $(S,p)$ by $r_{2k+1}^{-1}$, then 
	$$(r_{2k+1}^{-1}S,p)\overset{GH}\longrightarrow ([-1,\infty),0)$$
	because $r_{i+1}/r_i\to\infty$. In other words, $S$ has a tangent cone at infinity as the half line, but the base point does not correspond to the end point in this half line. 
\end{exmp}

\end{document}